\documentclass{article}
\usepackage{amsmath}
\usepackage{harvard}

\newtheorem{theorem}{Theorem}

\newenvironment{proof}[1][Proof]{\noindent\textbf{#1.} }{\ \rule{0.5em}{0.5em}}

\input{tcilatex}

\begin{document}

\title{A Note On The Spectral Norms of The Matrices Connected Integer
Numbers Sequence}
\author{Durmu\c{s} Bozkurt \and Department of Mathematics, Science Faculty
of Sel\c{c}uk University}
\maketitle

\begin{abstract}
In this paper, we compute the spectral norms of the matrices related with
integer squences and we give two examples related with Fibonacci and Lucas
numbers.
\end{abstract}


AMS Classification Number: 15A60, 15F35, 15B36, 15B57

\section{Introduction}

\bigskip In [1], the upper and lower bounds for the spectral norms of $r-$%
circulant\linebreak matrices are obtained by Shen and Cen. The lower bounds
for the norms of Cauchy-Toeplitz and Cauchy-Hankel matrices are given by Wu
in [2]. In [3-5], Solak and Bozkurt have found some bounds for the norms of
Cauchy-Toeplitz, Cauchy-Hankel and circulant matrices. In [5], Barrett and
Feinsilver have\linebreak defined the principal 2- minors and have
formulated the inverse of tridiagonal matrix.

Let $A$ be any $n\times n$ complex matrix. The well known spectral norm of
the matrix $A$ is%
\begin{equation*}
\left\Vert A\right\Vert _{2}=\sqrt{\max_{1\leq i\leq n}\left\vert \lambda
_{i}(A^{H}A)\right\vert }
\end{equation*}%
where $\lambda _{i}(A^{H}A)$ is eigenvalue of $A^{H}A$\ and $A^{H}$ is
conjugate transpose of the matrix $A$. $k$-principal minor of the matrix $A$
is denoted byare 
\begin{equation}
A\left( 
\begin{array}{c}
i_{1}i_{2}\ldots i_{k} \\ 
i_{1}i_{2}\ldots i_{k}%
\end{array}%
\right) =\left\vert 
\begin{array}{cccc}
a_{i_{1},i_{1}} & a_{i_{1},i_{2}} & \ldots & a_{i_{1},i_{k}} \\ 
a_{i_{2},i_{1}} & a_{i_{2},i_{2}} & \ldots & a_{i_{2},i_{k}} \\ 
\vdots & \vdots & \ddots & \vdots \\ 
a_{i_{k},i_{1}} & a_{i_{k},i_{2}} & \ldots & a_{i_{k},i_{k}}%
\end{array}%
\right\vert  \label{1}
\end{equation}%
where $1\leq i_{1}<i_{2}<\ldots <i_{k}\leq n$ $(1\leq k\leq n)[5].$

Now we define our matrix. $x_{i}$s are any integer numbers squence
for\linebreak $i=1,2,\ldots ,n.$ Let matrix $A_{x}$ be following form: 
\begin{equation}
A_{x}=[a_{ij}]_{i,j=1}^{n}=[x_{i}-x_{j}]_{i,j=1}^{n}  \label{2}
\end{equation}%
Obviously, $A_{x}$\ is skew-symmetric matrix. i.e. $A_{x}^{T}=-A_{x}.$ Since
eigenvalues of the skew-hermitian matrix $A_{x}$\ are pure imaginary,
eigenvalues of the matrix $iA_{x}$ are real where $i$ is complex unity.

\section{Main Result}

\begin{theorem}
Let the matrices $A_{x}$ be as in (\ref{1}). Then%
\begin{equation}
\left\Vert A_{x}\right\Vert _{2}=\dsum\limits_{1\leq r<s\leq
n}(x_{r}-x_{s})^{2}.  \label{3}
\end{equation}%
where $n\geq 4.$
\end{theorem}

\begin{proof}
If we substract $(i-1)$th row from $i$th row of the matrix $A_{x}$
for\linebreak $i=n,n-1,\ldots ,2,$ then we obtain%
\begin{equation*}
B_{x}=\left[ 
\begin{array}{ccccc}
0 & x_{1}-x_{2} & x_{1}-x_{3} & \ldots & x_{1}-x_{n} \\ 
x_{2}-x_{1} & x_{2}-x_{1} & x_{2}-x_{1} & \ldots & x_{2}-x_{1} \\ 
x_{3}-x_{2} & x_{3}-x_{2} & x_{3}-x_{2} & \ldots & x_{3}-x_{2} \\ 
\vdots & \vdots & \vdots & \ddots & \vdots \\ 
x_{n-1}-x_{n-2} & x_{n-1}-x_{n-2} & x_{n-1}-x_{n-2} & \ldots & 
x_{n-1}-x_{n-2} \\ 
x_{n}-x_{n-1} & x_{n}-x_{n-1} & x_{n}-x_{n-1} & \ldots & x_{n}-x_{n-1}%
\end{array}%
\right] .
\end{equation*}%
Obviously, $rank(B_{x})=rank(A_{x})=2.$ Since the matrix $A_{x}$ is
skew-symmetric, the matrix $iA_{x}$ is symmetric where $i$ is complex unity.
Then all the eigenvalues of the matrix $iA_{x}$ are real numbers. Moreover, $%
rank(A_{x})=rank(iA_{x}).$ Since determinants of all $k$-square submatrices
of the matrix $iA_{x}$ are zero for $k\geq 3$, all principal $k$-minors of
the matrix $iA_{x}$ are zero for $k\geq 3.$ Then characteristic polynomial
of the matrix $iA_{x}$%
\begin{equation*}
\Delta _{iA_{x}}(\lambda )=\lambda ^{n}+a_{1}\lambda ^{n-1}+a_{2}\lambda
^{n-2}
\end{equation*}%
where $(-1)^{k}a_{k}$ is the sum of principal $k$-minors of the matrix $%
iA_{x}$ where $k=1,2.$ On the other hand since $rank(iA_{x})=2,$ two
eigenvalues of the matrix $iA_{x}$ are nonzero. If $i\lambda $ is the
eigenvalue of the matrix $A_{x},$ then $-i\lambda $ is an eigenvalue of $%
A_{x}.$ Then $a_{1}=-tr(A_{x})=-tr(iA_{x})=\tsum\nolimits_{k=1}^{n}\lambda
_{k}=0$ where $\lambda _{k}$ are the eigenvalues of the matrix $iA_{x}.$
Cofficient $a_{2}$ is the sum of principal $2$-minors of any square matrix $%
A_{x}.$ i.e.%
\begin{equation*}
a_{2}=\dsum\limits_{1\leq r<s\leq n}A\left( 
\begin{array}{c}
r\ \ s \\ 
r\ \ s%
\end{array}%
\right) .
\end{equation*}%
Then we have%
\begin{eqnarray*}
a_{2} &=&\dsum\limits_{1\leq r<s\leq n}iA_{x}\left( 
\begin{array}{c}
r\ \ \ s \\ 
r\ \ s%
\end{array}%
\right) =\dsum\limits_{1\leq r<s\leq n}\left\vert 
\begin{array}{rr}
i(x_{r}-x_{r}) & i(x_{r}-x_{s}) \\ 
i(x_{s}-x_{r}) & i(x_{s}-x_{s})%
\end{array}%
\right\vert \\
&=&\dsum\limits_{1\leq r<s\leq n}\left\vert 
\begin{array}{cc}
0 & i(x_{r}-x_{s}) \\ 
-i(x_{r}-x_{s}) & 0%
\end{array}%
\right\vert =-\dsum\limits_{1\leq r<s\leq n}(x_{r}-x_{s})^{2}.
\end{eqnarray*}%
Hence we obtain%
\begin{equation*}
\Delta _{iA_{x}}(\lambda )=\lambda ^{n}-\left( \dsum\limits_{1\leq r<s\leq
n}(x_{r}-x_{s})^{2}\right) \lambda ^{n-2}.
\end{equation*}%
Then%
\begin{equation*}
\left\Vert iA_{x}\right\Vert _{2}=\left\Vert A_{x}\right\Vert
_{2}=\dsum\limits_{1\leq r<s\leq n}(x_{r}-x_{s})^{2}.
\end{equation*}%
The proof is completed.\bigskip
\end{proof}

\section{\protect\bigskip Numerical Examples}

The well known $F_{n}$ is $n$th Fibonacci number with recurence
relation\linebreak $F_{n}=F_{n-1}+F_{n-2}$ initial condition $F_{0}=0$ and $%
F_{1}=1$ and $L_{n}$ is $n$th\linebreak Lucas number with recurence relation 
$L_{n}=L_{n-1}+L_{n-2}$ initial condition $L_{0}=2$ and $L_{1}=1.$ The
matrices $F$ and $L$ are following forms:%
\begin{equation*}
F=[F_{i}-F_{j}]_{i,j=1}^{n}
\end{equation*}%
and%
\begin{equation*}
L=[L_{i}-L_{j}]_{i,j=1}^{n}.
\end{equation*}%
Let $\alpha =\dfrac{1+\sqrt{5}}{2}$ and $\beta =\dfrac{1-\sqrt{5}}{2}.$Then%
\begin{equation*}
F_{n}=\dfrac{\alpha ^{n}-\beta ^{n}}{\alpha -\beta }
\end{equation*}%
and%
\begin{equation*}
L_{n}=\alpha ^{n}+\beta ^{n}
\end{equation*}%
where $F_{n}$ and $L_{n}$ are $n$th Fibonacci and Lucas numbers,
respectively.\linebreak Furthermore%
\begin{equation*}
L_{n}=F_{n-1}+F_{n+1}.
\end{equation*}%
Now, we compute the spectral norms of the matrices $F$ and $L$. By the
definition of the spectral norm, we have

\begin{equation*}
\left\Vert F\right\Vert _{2}=\dsum\limits_{1\leq r<s\leq
n}(F_{r}-F_{s})^{2}=(n-1)F_{n+1}F_{n}-2\dsum\limits_{r=1}^{n-1}\dsum%
\limits_{s=r+1}^{n}F_{r}F_{s}.
\end{equation*}%
By the relationship between Fibonacci and Lucas numbers, we obtain%
\begin{equation*}
\left\Vert F\right\Vert _{2}=\left\{ 
\begin{array}{c}
(n-1)F_{n+1}F_{n}-\dfrac{2}{5}\left(
L_{n}-2+\dsum\limits_{r=1}^{n-1}\dsum\limits_{s=r+1}^{n}L_{r+s}\right) ,\ n\
is\ even \\ 
(n-1)F_{n+1}F_{n}-\dfrac{2}{5}\left(
L_{n}-1+\dsum\limits_{r=1}^{n-1}\dsum\limits_{s=r+1}^{n}L_{r+s}\right) ,\ n\
is\ odd%
\end{array}%
\right. .
\end{equation*}%
Similarly,%
\begin{equation*}
\left\Vert L\right\Vert _{2}=\dsum\limits_{1\leq r<s\leq
n}(L_{r}-L_{s})^{2}=(n-1)(L_{n+1}L_{n}-2)-2\dsum\limits_{r=1}^{n-1}\dsum%
\limits_{s=r+1}^{n}L_{r}L_{s}.
\end{equation*}%
Then%
\begin{equation*}
\left\Vert L\right\Vert _{2}=\left\{ 
\begin{array}{c}
(n-1)(L_{n+1}L_{n}-2)-2\left(
L_{n}-2+\dsum\limits_{r=1}^{n-1}\dsum\limits_{s=r+1}^{n}L_{r+s}\right) ,\ n\
is\ even \\ 
(n-1)(L_{n+1}L_{n}-2)-2\left(
L_{n}-1+\dsum\limits_{r=1}^{n-1}\dsum\limits_{s=r+1}^{n}L_{r+s}\right) ,\ n\
is\ odd%
\end{array}%
\right. .
\end{equation*}

\end{document}